\documentclass{article}
\usepackage[utf8]{inputenc}
\usepackage{amsmath}
\usepackage{amssymb}
\usepackage{amsthm}
\usepackage[margin=1 in]{geometry}
\usepackage{parskip}
\usepackage{hyperref}

\newtheorem{lem}{Lemma}
\newtheorem{thm}{Theorem}

\theoremstyle{definition}
\newtheorem{definition}{Definition}

\title{A New Result Regarding Descartes Numbers}
\author{Pratik Rathore}
\date{\today}

\begin{document}

\maketitle

\section{Introduction}
In the field of number theory, a positive integer $n$ is called a $\textit{perfect number}$ if it satisfies the relation $\sigma(n) = 2n$. Over the years, many even perfect numbers have been found, but it is still not known whether there are any odd perfect numbers. It is then natural to ask whether we can find odd positive integers $n$ that are ``almost perfect." 

For example, consider the number $\mathcal{D} = 3^{2}7^{2}11^{2}13^{2}22021$, also known as the number of Descartes. If one (incorrectly) assumes that $22021 = 19^{2} \cdot 61$ is prime, we find that 

\[\sigma(\mathcal{D}) = \sigma(3^{2}7^{2}11^{2}13^{2}22021) = \sigma(3^{2}7^{2}11^{2}13^{2})(22021+1) = 2\mathcal{D}\]

Inspired by this example, we define a family of numbers, known as Descartes numbers, as follows:

\begin{definition}
An odd positive integer $n$ is called a $\textit{Descartes number}$ if there exist positive integers $k,m$ such that $n = km$ and 

\begin{equation}
\sigma(k)(m+1) = 2km \label{eq:ogdef}
\end{equation}

\end{definition}

Setting $k = 3^{2}7^{2}11^{2}13^{2}$ and $m = 22021$ in the above definition, we see that $\mathcal{D}$ is a Descartes number. In fact, $\mathcal{D}$ is the only known Descartes number.

For simplicity, only cube-free Descartes numbers (defined below) are investigated in this paper.

\begin{definition}
A Descartes number $n$ is called a cube-free Descartes number if $\nexists$ a prime $p$ such that $p^{3} \mid n$.
\end{definition}

In $2008$, Banks et al. $\cite{banks}$ showed that $\mathcal{D}$ is the only cube-free Descartes number with fewer than seven distinct prime divisors. In this paper, we will prove the following theorem:

\begin{thm}
There is no cube-free Descartes number with seven distinct prime factors.
\end{thm}

This theorem implies that $\mathcal{D}$ is the only cube-free Descartes number with fewer than eight distinct prime divisors.

\section{Preparations}
Before we begin to prove the theorem below, we present another definition and several lemmas.\\

\begin{definition}
For a prime $p$, let $f(p) = \dfrac{p^{2}}{\sigma(p^{2})}$.
\end{definition}

\begin{lem}
(Nielsen \cite{nielsen}) Every odd perfect number has at least nine distinct prime divisors. 
\end{lem}

In \cite{banks}, the authors base many of their proofs on perfect numbers having at least seven distinct prime divisors. Having Lemma $1$ allows us to extend the methods of Banks et al. to the subsequent proofs given in this paper.

The following six lemmas are all proven in \cite{banks}.

\begin{lem}
If $n = km$ is a cube-free Descartes number with $3 \mid n$, $\#\{p : p \mid k$ and $p \equiv 1 \pmod{3}\} = 2$ and $3 \mid k$.
\end{lem}

\begin{lem}
Let $\omega(n)$ be the number of distinct prime factors of $n$. If $n = km$ is a cube-free Descartes number with $3 \mid n$, then $\omega(k) \geq 4$.
\end{lem}

\begin{lem}
If $p$ and $q$ are primes such that $p^{2} + p + 1 \equiv 0 \pmod{q}$, then $q = 3$ or $q \equiv 1 \pmod{3}$. If $s$ is square-free, then the number $\sigma(s^{2})$ has no prime divisor $q \equiv 2 \pmod{3}$.
\end{lem}

\begin{lem}
If $n = km$ is a cube-free Descartes number, then $m \equiv 1 \pmod{12}$ and $\gcd(k,m) = 1$.
\end{lem}

\begin{lem}
If $n = km$ is a Descartes number, then $k = s^{2}$ for some positive squarefree integer $s$ and $m \mid \sigma(s^{2})$.
\end{lem}

\begin{lem}
If $n = km$ is a Descartes number with $3 \nmid n$, then $n$ has more than one million distinct prime divisors.
\end{lem}

Since we are only working with $n$ that have seven distinct prime factors, Lemma $7$ guarantees that $3 \mid n$, and will assume this throughout the rest of the paper.

We will also need the following two results to prove the new theorem.

\begin{lem}
$f(p) < 1$ for all primes $p$.
\end{lem}

\begin{proof}
For all primes $p$, it is clear that $p^2 < p^2 + p + 1 = \sigma(p^{2})$. The result follows.
\end{proof}

\begin{lem}
If $n = km$ is a cube-free Descartes number with fewer than nine distinct prime factors, then $m \geq 49$.
\end{lem}

\begin{proof}
Using Lemmas $1$ and $5$, we observe that $m \equiv 1 \pmod{12}$ cannot be prime. If $m = 1$, we use the definition of a Descartes number and find that $\sigma(k) = k$, contradiction. If $m = 25$, we apply Lemmas $6$ and $4$ to reach a contradiction. Thus $m \geq 49$.
\end{proof}

\section{The Key Lemma}

\begin{lem}
If a cube-free Descartes number $n = km$ has exactly $7$ distinct prime factors, then $5 \nmid k$.
\end{lem}

\begin{proof}
On the contrary, suppose $5 \mid k$. Since $n$ has $7$ distinct prime factors, and $\gcd(k,m) = 1$ (by Lemma $5$), we have $\omega(k) + \omega(m) = 7$. By Lemma $3$, we have three distinct cases to consider:\\

\textbf{Case 1: $\omega(m) = 1, \omega(k) = 6$}\\
In their paper, Banks et al. show that if $\omega(m) = 1, \omega(k) = 4,5$, then there is no such $k$ such that $5 \mid k$. Using Lemma $1$, we can easily extend their proof and conclude that this result also holds true for $\omega(m) = 1, \omega(k) = 6,7$.

\textbf{Case 2: $\omega(m) = 2, \omega(k) = 5$}\\
In their paper, Banks et al. show that if $\omega(m) = 2, \omega(k) = 4$ and we assume that $5 \mid k$, then either $13 \mid m$ and $31 \mid k$ or $31 \mid m$ and $13 \mid k$. Using Lemma $1$, we can easily extend their proof and conclude that this result also holds true for $\omega(m) = 2, \omega(k) = 5,6$.

If $13 \mid m$, then $m \equiv 949 \pmod{3900}$, meaning that $m \geq 949$. In addition, we have that $31 \mid k$, implying that $k = 3^{2}5^{2}31^{2}p^{2}q^{2}$. WLOG assume $p < q$.

We obtain the inequality 

\[\frac{2 \cdot 949}{950} \leq \frac{2m}{m+1} = \frac{\sigma(k)}{k} = \frac{\sigma(3^{2}5^{2}31^{2})}{3^{2}5^{2}31^{2}}\frac{1}{f(p)}\frac{1}{f(q)} < 2 \].

Upon further calculation, it can be seen that the possible values for the ordered pair $(p,q)$ are 

\[\{(17, 59), (17, 61), (19, 43), (23, 31)\}\]

However, by Lemma $2$, we can throw out $(17, 59)$ and $(19, 43)$. We also throw out $(23, 31)$ since $k$ must be cube-free. The remaining pair, $(17, 61)$ does not yield an integer solution for $m$.  

If $31 \mid m$, then $m \equiv 2449 \pmod{3900}$, meaning that $m \geq 2449$. In addition, we have that $13 \mid k$, implying that $k = 3^{2}5^{2}13^{2}p^{2}q^{2}$. WLOG assume $p < q$.

We obtain the inequality 

\[\frac{2 \cdot 2449}{2450} \leq \frac{2m}{m+1} = \frac{\sigma(k)}{k} = \frac{\sigma(3^{2}5^{2}13^{2})}{3^{2}5^{2}13^{2}}\frac{1}{f(p)}\frac{1}{f(q)} < 2 \]

Upon further calculation, it can be seen that the possible values for the ordered pair $(p,q)$ are 

\[\{(37, 307), (37, 311), (37, 313), (37, 317), (37, 331), (37, 337), (41, 167),\] 
\[(41, 173), (47, 109), (47, 113), (53, 89), (61, 71), (61, 73)\}\]

By Lemma $2$, we can throw out all pairs except for $(37, 311), (37, 317), (47, 109), (61, 71)$. None of these pairs yield an integer solution for $m$.

\textbf{Case 3: $\omega(m) = 3, \omega(k) = 4$}

Since $\omega(k) = 4$, $k = 3^{2}5^{2}p^{2}q^{2}$, with $p \neq q$ and $p,q \equiv 1 \pmod{3}$. WLOG let $p < q$.

When this expression for $k$ is plugged into $\eqref{eq:ogdef}$, we see that

\begin{equation}
13 \cdot 31 \sigma(p^{2}) \sigma(q^{2}) (m+1) = 2 \cdot 3^{2}5^{2}p^{2}q^{2}m \label{eq:lem1.1}    
\end{equation}

We will first show $(13 \cdot 31) \mid m$.

Suppose $13 \mid p$, meaning that $p = 13$. By substituting this value into $\eqref{eq:lem1.1}$, we find that 

\begin{equation}
31 \cdot 61 \sigma(q^{2}) (m+1) = 2 \cdot 3 \cdot 5^{2} \cdot 13 \cdot q^{2} \cdot m \label{eq:lem1.2}    
\end{equation}

Rearranging $\eqref{eq:lem1.2}$, we obtain the following equation:

\[\frac{2m}{m+1} = \frac{31 \cdot 61}{3 \cdot 5^{2} \cdot 13 \cdot f(q)}\]

By Lemma $9$, $m \geq 49$, so we see that

\[ \frac{2 \cdot 49}{50} \leq \frac{31 \cdot 61}{3 \cdot 5^{2} \cdot 13 \cdot f(q)} < 2\]

The only solutions for $q$ that are primes and $\equiv 1 \pmod{3}$ are $q = 37, 43, 61, 67, 73, 79$. None of these values of $q$ give an integer solution for $m$. Thus $13 \nmid p$.

If we assume $31 \mid p$, we use the same method to arrive at the inequality

\[\frac{2 \cdot 49}{50} \leq \frac{13 \cdot 331}{3 \cdot 5^{2} \cdot 31 \cdot f(q)} < 2\]

However, this inequality gives no solutions for $q$ that satisfy the required conditions. Thus $31 \nmid p$.

Since $p$ and $q$ are arbitrary, we have also shown that neither $13$ nor $31$ divides $q$. Therefore, $(13 \cdot 31) \mid m$, meaning that $m = 13^{a}31^{b}t$, where $1 \leq a,b \leq 2$ and $t \geq 7$ is a prime or a square of a prime.

Using the expression for $m$ in $\eqref{eq:ogdef}$, we find that

\[ 13 \cdot 31 \sigma(p^{2}) \sigma(q^{2}) (13^{a} \cdot 31^{b} \cdot t + 1) = 2 \cdot 3^{2}5^{2}p^{2}q^{2} \cdot 13^{a} \cdot 31^{b} \cdot t\]

\[\sigma(p^{2}) \sigma(q^{2}) (13^{a} \cdot 31^{b} \cdot t + 1) = 2 \cdot 3^{2}5^{2}p^{2}q^{2} \cdot 13^{a-1} \cdot 31^{b-1} \cdot t\]

\begin{equation}
\frac{13^{a} \cdot 31^{b} \cdot t + 1}{2 \cdot 3^{2}5^{2} \cdot 13^{a-1} \cdot 31^{b-1} \cdot t} = f(p)f(q) \label{eq:lem1.3}    
\end{equation}

We observe that 

\[\frac{403}{450} \leq \frac{13^{a} \cdot 31^{b} \cdot t + 1}{2 \cdot 3^{2}5^{2} \cdot 13^{a-1} \cdot 31^{b-1} \cdot t} \leq \frac{13^{1} \cdot 31^{1} \cdot 7 + 1}{2 \cdot 3^{2}5^{2} \cdot 13^{0} \cdot 31^{0} \cdot 7} = \frac{2822}{3150}\]

and so using $\eqref{eq:lem1.3}$

\[\frac{403}{450} \leq f(p)f(q) \leq \frac{2822}{3150}\]

Since $p \equiv 1 \pmod{3}$ and $n$ is cube-free, the two smallest values for $p$ are $7$ and $19$. 

Since $f(p)$ increases as $p$ increases, and $f(19) > \sqrt{\dfrac{2822}{3150}}$, we must have $p = 7$. However, $f(7) < \dfrac{403}{450}$. Combined with the fact that $f(q)$ is never greater than $1$ (by Lemma $8$), we find that there is no corresponding value of $q$ satisfying the inequality.

Since we have shown that $5 \nmid k$ in all $3$ cases, we have proven the desired statement.

\end{proof}

\section{Proof of  Theorem}

\begin{proof}
By Lemma $3$, we have three distinct cases to consider:

\textbf{Case 1: $\omega(m) = 1, \omega(k) = 6$}

We see that at least one of $7$ or $11$ must divide $k$, otherwise

\[ \frac{2m}{m+1} = \frac{\sigma(k)}{k} \leq \frac{\sigma(3^{2}13^{2}17^{2}19^{2}23^{2}29^{2})}{3^{2}13^{2}17^{2}19^{2}23^{2}29^{2}}\]

is impossible for $m \geq 49$. 

If $7 \mid k$, we have that $k = 3^{2}7^{2}l^{2}$, where $l$ is square-free. Then by $\eqref{eq:ogdef}$ we have

\[13 \cdot 19\sigma(l^{2})(m+1) = 2 \cdot 3 \cdot 7^{2} l^{2} m\] 
By Lemma $2$, we have that at least one of $13$ or $19$ must divide $m$, meaning that $m$ is either $13^{2}$ or $19^{2}$. If $m = 13^{2}$, then $5 \mid (m+1)$, meaning that $5 \mid k$, but this is impossible by Lemma $10$. If $m = 19^{2}$, then we have $13 \mid l \implies 13 \mid k$ and $181 \mid (m+1) \implies 181 \mid k$, which contradicts Lemma $2$. Thus $7 \nmid k$.

If $11 \mid k$, we have that $k = 3^{2}{11}^{2}l^{2}$, where $l$ is square-free. Then by $\eqref{eq:ogdef}$ we have

\[13 \cdot 7 \cdot 19 \sigma(l^{2}) (m+1) = 2 \cdot 3^{2} \cdot 11^{2} \cdot l^{2} \cdot m\]

By Lemma $2$, we see that not all of $7$, $13$, and $19$ cannot divide $l$, therefore $m$ must be one of $7^{2}$, $13^{2}$, or $19^{2}$. In each case, $11^{2} \nmid (m+1)$. Therefore, $11 \mid \sigma(l^{2})$, but this is impossible by Lemma $4$.

Thus, it is impossible for a cube-free Descartes number with $7$ distinct prime factors to have $\omega(m) = 1$.\\

\textbf{Case 2: $\omega(m) = 2, \omega(k) = 5$}

We see that $7 \mid k$, otherwise

\[ \frac{2m}{m+1} = \frac{\sigma(k)}{k} \leq \frac{\sigma(3^{2}11^{2}13^{2}17^{2}19^{2})}{3^{2}11^{2}13^{2}17^{2}19^{2}}\]

is impossible for $m \geq 49$.

Then by $\eqref{eq:ogdef}$ we have

\begin{equation}
13 \cdot 19\sigma(l^{2})(m+1) = 2 \cdot 3 \cdot 7^{2} l^{2} m \label{eq:case2key1} 
\end{equation}

By Lemma $2$, we see that $(13 \cdot 19) \nmid l$, so at least one of $13$ or $19$ divides $m$.

Now suppose that $(13 \cdot 19) \mid m$. Using Lemma $5$, the fact that $m$ is cube-free, and $\omega(m) = 2$, we see that $m = 13 \cdot 19^{2}$ or $m = 13^{2} \cdot 19^{2}$. 

If $m = 13 \cdot 19^{2}$, we substitute into $\eqref{eq:case2key1}$ to obtain

\begin{equation}
2347\sigma(l^{2}) = 3 \cdot 7^{2} 19 l^{2} \label{eq:case2-1.1}
\end{equation}

Then it must be true that $l = 2347pq$ for distinct primes $p,q$. Substituting this expression into $\eqref{eq:case2-1.1}$, we find that

\begin{equation}
397 \cdot 661 \sigma(p^{2})\sigma(q^{2}) = 7 \cdot 19 \cdot 2347p^{2}q^{2} \label{eq:case2-1.2}
\end{equation}

Since $397$ and $661$ are prime, $\eqref{eq:case2-1.2}$ implies that $p,q$ are $397, 661$ in some order. However, these values for $p$ and $q$ do not satisfy the equation above. 

If $m = 13^{2} \cdot 19^{2}$, we substitute into $\eqref{eq:case2key1}$ to obtain

\begin{equation}
(13^{2} \cdot 19^{2} + 1)\sigma(l^{2}) = 2 \cdot 3 \cdot 7^{2} \cdot 13 \cdot 19 l^{2} \label{eq:case2-1.3}
\end{equation}

Since $5 \mid (13^{2} \cdot 19^{2} + 1)$, we have by $\eqref{eq:case2-1.3}$ that $5 \mid l \implies 5 \mid k$, which contradicts Lemma $10$.

Therefore, $(13 \cdot 19) \nmid m$.

Now suppose that $19 \mid l$ and $13 \mid m$, meaning that $l = 19pq$ for primes $p,q$; note that $p,q \equiv 2 \pmod{3}$ by Lemma $2$. Then by $\eqref{eq:case2key1}$ we have

\begin{equation}
13 \cdot 127 \sigma(p^{2})\sigma(q^{2})(m + 1) = 2 \cdot 7^{2} 19 p^{2}q^{2}m \label{eq:case2key2}
\end{equation}

Since we have $p,q \equiv 2 \pmod{3}$, it must be true that $(13 \cdot 127) \mid m$. Since $m$ is cube-free, $\omega(m) = 2$, and $m \equiv 1 \pmod{12}$, we see that $m = 13 \cdot 127^{2}$ or $m = 13^{2}127^{2}$.

If $m = 13 \cdot 127^{2}$, $\eqref{eq:case2key2}$ becomes

\begin{equation}
17 \cdot 881 \sigma(p^{2})\sigma(q^{2}) = 7 \cdot 19 \cdot 127 p^{2}q^{2} \label{eq:case2-2.1}
\end{equation}

Since $17$ and $881$ are prime, we have that $p,q$ are $17, 881$ in some order. However, these values for $p$ and $q$ do not satisfy $\eqref{eq:case2-2.1}$.

Then we assume that $m = 13^{2}127^{2}$, and by plugging this value into $\eqref{eq:case2key2}$ we obtain

\begin{equation}
397 \cdot 3433 \sigma(p^{2})\sigma(q^{2}) = 7^{2} \cdot 13 \cdot 19 \cdot 127 p^{2}q^{2} \label{eq:case2-2.2}
\end{equation}

By the same method as the previous case, we conclude that there are no $p,q$ satisfying $\eqref{eq:case2-2.2}$.

We must now consider the final case, $13 \mid l$ and $19 \mid m$. Since $13 \mid l$, we have that $l = 13pq$ for some primes $p,q$, with $p,q \equiv 2 \pmod{3}$ by Lemma $2$. 

By substituting this expression for $l$ into $\eqref{eq:case2key1}$, we find that

\begin{equation}
19 \cdot 61 \sigma(p^{2})\sigma(q^{2})(m + 1) = 2 \cdot 7^{2} 13 p^{2}q^{2}m \label{eq:case2-3.1}   
\end{equation}

Since $p, q \equiv 2 \pmod{3}$,  we must have that $(19 \cdot 61) \mid m$. Since $m$ is cube-free, $\omega(m) = 2$, and $m \equiv 1 \pmod{12}$, we see that $m = 19^{2} \cdot 61$ or $m = 19^{2}61^{2}$.

If $m = 19^{2}61^{2}$, by $\eqref{eq:case2-3.1}$ we have that

\begin{equation}
337 \cdot 1993 \sigma(p^{2})\sigma(q^{2}) = 7^{2} \cdot 13 \cdot 19 \cdot 61 p^{2}q^{2} \label{eq:case2-3.2}    
\end{equation}

Then $p,q = 337, 1993$ in some order. One can verify that these values for $p, q$ do not satisfy $\eqref{eq:case2-3.2}$.

If $m = 19^{2} \cdot 61$, by $\eqref{eq:case2-3.1}$ we have that

\begin{equation}
11^{2} \sigma(p^{2})\sigma(q^{2}) = 7 \cdot 19 p^{2}q^{2} \label{eq:case2-3.3}    
\end{equation}

Clearly $11$ divides at least one of $p$ or $q$. $11$ cannot divide both $p$ and $q$, otherwise $k$ would not be cube-free. WLOG let $11 \mid p$, meaning that $p = 11$. Substituting into $\eqref{eq:case2-3.3}$, we find that

\[q^{2} + q + 1 = q^{2}\]

which is impossible. 

Thus, it is impossible for a cube-free Descartes number with 7 distinct prime factors to have $\omega(m) = 2$.

\textbf{Case 3: $\omega(m) = 3, \omega(k) = 4$}

We can again show that $7 \mid k$. Then by $\eqref{eq:ogdef}$ we have

\begin{equation}
13 \cdot 19\sigma(l^{2})(m+1) = 2 \cdot 3 \cdot 7^{2} l^{2} m \label{eq:case3key1}
\end{equation}

where $l = pq$ for distinct primes $p,q$. 

$\eqref{eq:case3key1}$ then becomes 

\begin{equation}
13 \cdot 19\sigma(p^{2})\sigma(q^{2})(m+1) = 2 \cdot 3 \cdot 7^{2} p^{2} q^{2} m \label{eq:case3key2}
\end{equation}

Applying Lemma $2$ we find that there are $3$ distinct cases: $13$ divides one of $p$ or $q$, $19$ divides one of $p$ or $q$, or neither $13$ nor $19$ divide $p$ or $q$. 

If $13$ divides one of $p$ or $q$, WLOG let $p = 13$. Substituting into $\eqref{eq:case3key2}$, we obtain

\begin{equation}
19 \cdot 61 \sigma(q^{2})(m+1) = 2 \cdot 7^{2} \cdot 13 \cdot q^{2}m \label{eq:case3-3.1}    
\end{equation}

We see that $(19 \cdot 61) \mid m \implies m \geq 1159$.

Then \[\frac{2 \cdot 1159}{1160} \leq \frac{2m}{m+1} = \frac{19 \cdot 61 \sigma(q^{2})}{7^{2} \cdot 13 \cdot q^{2}} < 2\]

We find that $q = 11$ is the only prime satisfying this inequality. Using $\eqref{eq:case3-3.1}$, we find $m = 19^{2} \cdot 61$, but then $w(m) = 2$, which contradicts our original assumption.\footnote{$m = 19^{2} \cdot 61$ gives the number of Descartes.}

Now suppose $19$ divides one of $p$ or $q$. WLOG let $p = 19$. Substituting into $\eqref{eq:case3key2}$, we obtain

\[ 13 \cdot 127 \sigma(q^{2})(m+1) = 2 \cdot 7^{2} \cdot 19 \cdot q^{2}m \]

We see that $(13 \cdot 127) \mid m \implies m \geq 1651$.

Then \[\frac{2 \cdot 1651}{1652} \leq \frac{2m}{m+1} = \frac{13 \cdot 127 \sigma(q^{2})}{7^{2} \cdot 19 \cdot q^{2}} < 2\]

but this inequality has no prime integer solutions.

We must now consider the case where neither $13$ nor $19$ divide $p$ or $q$. In this case, we see that both $13$ and $19$ must divide $m$, implying that $m = 13^{a}19^{b}t$, where $1 \leq a,b \leq 2$ and $t \geq 5$ is a prime or a square of a prime.

Substituting into $\eqref{eq:case3key2}$ and rearranging, we see that

\[ 13 \cdot 19 \sigma(p^{2}) \sigma(q^{2}) (13^{a} \cdot 19^{b} \cdot t + 1) = 2 \cdot 3 \cdot 7^{2}p^{2}q^{2} \cdot 13^{a} \cdot 19^{b} \cdot t\]

\[\sigma(p^{2}) \sigma(q^{2}) (13^{a} \cdot 19^{b} \cdot t + 1) = 2 \cdot 3 \cdot 7^{2}p^{2}q^{2} \cdot 13^{a-1} \cdot 19^{b-1} \cdot t\]

\begin{equation}
\frac{13^{a} \cdot 19^{b} \cdot t + 1}{2 \cdot 3 \cdot 7^{2} \cdot 13^{a-1} \cdot 19^{b-1} \cdot t} = f(p)f(q) \label{eq:case3-3.2}    
\end{equation}

We also observe that

\[\frac{247}{294} \leq \frac{13^{a} \cdot 19^{b} \cdot t + 1}{2 \cdot 3 \cdot 7^{2} \cdot 13^{a-1} \cdot 19^{b-1} \cdot t} \leq \frac{13 \cdot 19 \cdot 5 + 1}{2 \cdot 3 \cdot 7^{2} \cdot 5} = \frac{206}{245}\]

and so using $\eqref{eq:case3-3.2}$

\[\frac{247}{294} \leq f(p)f(q) \leq  \frac{206}{245}\]

Since $3 \mid k$, $p,q \geq 5$. WLOG let $p < q$. Since $f(17) > \sqrt{\dfrac{206}{245}}$ and $f(p)$ is increasing, we see that $p \in \{ 5, 7, 11\}$.

If $p = 5$, we have $f(5) < \dfrac{247}{294}$ and $f(q) < 1$, meaning that $f(5)f(q) < \dfrac{247}{294}$, so this is not possible.

If $p = 7$, we have 

\[ \frac{247}{294} \cdot \frac{1}{f(7)} \leq \frac{q^{2}}{\sigma(q^{2})} \leq \frac{206}{245} \cdot \frac{1}{f(7)}\]

It turns out that this inequality has no working values for $q$. For $p = 11$, we repeat the argument used for $p = 7$ to conclude that $p \neq 11$.

Since all three values for $p$ do not work, this case is also impossible.

Therefore, there is no cube-free Descartes number with seven distinct prime factors.

\end{proof}

Using Theorem $1$ and the previous results proven by Banks et al., we conclude that $\mathcal{D}$ is the only cube-free Descartes number with fewer than eight distinct prime factors.

\section{Acknowledgements}
The author would like to thank Dr. Larry Washington for his mentorship throughout the project.

\end{document}